\documentclass{amsart}
\usepackage[margin=10pt,font=small,labelfont=normalfont]{caption}
\usepackage{subcaption}

\usepackage{mathtools}

\usepackage[demo]{graphicx}

\usepackage{xcolor}

\usepackage[margin=1.1in]{geometry} 

\usepackage{graphics,graphicx}
\usepackage{pstricks,pst-node,pst-tree}

\usepackage{xcolor}
\theoremstyle{plain}

\usepackage{amsfonts}
\usepackage{amssymb}
\usepackage{amsthm}
\usepackage{amsmath}
\usepackage{enumerate}
\usepackage{hyperref}
\hypersetup{
    colorlinks=false,
    pdfborder={0 0 0},
}
\usepackage{amsrefs}
\usepackage{graphicx}
\usepackage{verbatim}
\usepackage{amscd}
\usepackage{mathtools}

\usepackage{arydshln}

\theoremstyle{plain}

\newcommand{\IC}{\mathbb{C}}

\def\sqr#1#2{{\,\vcenter{\vbox{\hrule height.#2pt\hbox{\vrule width.#2pt
height#1pt \kern#1pt\vrule width.#2pt}\hrule height.#2pt}}\,}}

\newtheorem{proposition}{Proposition}[section]
\newtheorem{lemma}[proposition]{Lemma}
\newtheorem{theorem}[proposition]{Theorem}
\newtheorem{corollary}[proposition]{Corollary}

\theoremstyle{definition}
\newtheorem{definition}[proposition]{Definition}
\theoremstyle{remark}
\newtheorem{question}{Question}
\newtheorem*{convention*}{Convention}
\newtheorem{example}[proposition]{Example}

\newtheorem{remark}[proposition]{Remark}

\numberwithin{equation}{section}

\newcommand{\calC}{\hbox{$\mathcal C$}}
\newcommand{\calA}{\hbox{$\mathcal A$}}
\newcommand{\calH}{\hbox{$\mathcal H$}}

\newcommand{\calD}{\hbox{$\mathcal D$}}

\numberwithin{equation}{section}

\begin{document}

\title{Anti-C*algebras} 


\author[Pluta]{Robert Pluta}
\email{plutar@tcd.ie}
\address{Department of Mathematics, Northeastern University Oakland, USA}

\author[Russo]{Bernard Russo}
\email{brusso@math.uci.edu}
\address{Department of Mathematics, UC Irvine, Irvine CA, USA}

\keywords{C*-ternary ring, ternary ring of operators, TRO,  normed standard embedding, Hilbert module, semisimple, anti-C*-algebra}

\date{November 18, 2023}

\maketitle

\begin{abstract}
 We introduce a class of Banach algebras that we call anti-C*-algebras.     We show that the normed standard embedding of a C*-ternary ring is the direct sum of a C*-algebra and an anti-C*-algebra. We prove that C*-ternary rings and anti-C*-algebras are semisimple. 
  We give two new characterizations of C*-ternary rings which are isomorphic to a TRO (ternary ring of operators), providing answers to a query raised by Zettl in 1983, and we propose some problems for further study.\end{abstract}


\section{Introduction}
One of the results of \cite{PluRusJMAA}, namely,  \cite[Prop.\  2.7]{PluRusJMAA}, is not accurately formulated as stated. It is not true that the normed standard embedding ${\mathcal A}(M)$ of a C*-ternary ring $M$  is always a  C*-algebra. 
The corrected version of Proposition 2.7 in \cite{PluRusJMAA}  states that ${\mathcal A}(M)$ is the direct sum of a C*-algebra (corresponding to the positive sub-C*-ternary ring $M_+$ of $M$) and a Banach algebra ${\mathcal B}$  (corresponding to the negative sub-C*-ternary ring $M_-$ of $M$).   This is made precise in Theorem~~\ref{prop:0212231} below. 
The algebra ${\mathcal B}$ is a Banach algebra with a continuous involution and an approximate identity, and exploring the  properties of this class of algebras is a topic worthy of further study.  We embark on that study in this paper by showing, among other things, that it is semisimple.  

This correction affects only  the results of \cite[Section 4]{PluRusJMAA}.
The necessary adjustments for \cite[Section 4]{PluRusJMAA} are  provided in detail  in section~\ref{sec:0706231} below. 
All other results of \cite{PluRusJMAA} remain valid as stated.

In this paper  we shall call  the algebra $\mathcal B$ an {\it anti-C*-algebra}, since on the one hand, it is derived from an anti-TRO in the same way that the linking C*-algebra is derived from a TRO (Definition~\ref{def:0422231}). Another justification is that, although by virtue of the semisimplicity it has a unique complete  algebra norm topology, it cannot be renormed to be a C*-algebra.

In section~\ref{sec:0410231}, we introduce the algebra $\mathcal B$ via a dense *-subalgebra ${\mathcal B}_0$ (Proposition~\ref{prop:0209231}), and in addition to Theorem~~\ref{prop:0212231}, we give two new characterizations of when a C*-ternary ring is isomorphic to a TRO (Theorem~\ref{thm:0413231}).
In section~\ref{sec:0410232} we prove the semisimplicity of $\mathcal B$ (Theorem~\ref{thm:0414231}). In section~\ref{sec:0410233}  we state a relation between the ideals of a C*-ternary ring and those of its normed standard embedding (Proposition~\ref{prop:0410233}), 
and include 
a  list of questions for further study. In section~\ref{sec:0706231}, as noted above,  we revisit the topic of ternary operator categories, indicating the necessary adjustments to \cite[Section 4]{PluRusJMAA}.  

Two distinct kinds of associative triple systems are recognized in the literature: those of the first kind and those of the second kind. The first one has been studied by Lister \cite{Lister1971} but their mention here is solely for the purpose of completeness and will not be exploited. Instead, our attention will be directed exclusively toward examining the latter. An associative triple system of the second kind (hereafter, an associative triple system)  is defined as a complex vector space $M$ equipped with a function $(x, y, z) \mapsto [x y z]$ from $M \times M \times M$ to $M$ that exhibits linearity with respect to the outer variables, conjugate-linearity with respect to the middle variable, and satisfies  
\[
[[xyz]uv]
=
[x[u z y]v]
=
[xy[zuv]]
\]
for all $x, y, z, u, v$ in $M$. 

 These systems have been studied by Hestenes \cite{Hestenes1962}, Loos \cite{Loos1972}, Meyberg \cite{Meyberg72}, and others, for example, \cite{Abadie2017, LanRus83, Myung75, NeaRus03, PluRusJMAA, Zettl83}.
Recall that a {\it C*-ternary ring} was introduced by Zettl\footnote{Which he called a {\it ternary C*-ring}} in \cite{Zettl83} as an associative triple system $M$ which is also a complex Banach space for which 
$$\|[x,y,z]\|\le \|x\|\|y\|\|z\| \hbox{ and } \|[x,x,x]\|=\|x\|^3.
$$ In addition, if $M$ is a dual Banach space, it is called a {\it W*-ternary ring}.

\section{The normed standard embedding of a C*-ternary ring}\label{sec:0410231}

In what follows, we shall use some notation and some results from  \cite{PluRusJMAA}, making precise references to  \cite{PluRusJMAA} when necessary.
For the convenience of the reader, we first recall the construction of the normed standard embedding.

Let $M$ be an associative triple system with triple product denoted by $[hgf]$. Let 
\[
E(M)=\hbox{End}\, (M)\oplus[\overline{ \hbox{End}\, (M)}]^{op},
\]
where the notation $\overline{V}$ for a complex vector space means that the scalar multiplication in $\overline{V}$ is $(\lambda,v)\in \IC\times V\mapsto \lambda\circ v=\overline{\lambda}v$.  
We shall 
denote the products in $E(M)^{op}$ and in $[\hbox{End}\, (M)]^{op}$ by $X\circ Y:=YX$.\smallskip

Involutions 
are defined on $E(M)$  by 
$
A=(A_1,A_2)\mapsto \overline{A}=\overline{(A_1,A_2)}=(A_2,A_1),
$
and on $E(M)^{op}$ by
$
B=(B_1,B_2)\mapsto \overline{B}=\overline{(B_1,B_2)}=(B_2,B_1).
$
For $g,h\in M$, define
\[
L(g,h)=[gh\cdot], R(g,h)=[\cdot hg],
\]
\[
\ell(g,h)=(L(g,h), L(h,g))=([gh\cdot],[hg\cdot])\in E(M),
\]
\[
r(g,h)=(R(h,g), R(g,h))=([\cdot gh],[\cdot hg])\in E(M)^{op},
\]
\[
L_0=L_0(M)=\hbox{span}\, \{\ell(g,h): g,h\in M\}\subset E(M),
\]
\[
R_0=R_0(M)=\hbox{span}\, \{r(g,h): g,h\in M\}\subset E(M)^{op}.
\]

Let  $A=(A_1,A_2)\in E(M)$, $B=(B_1,B_2)\in E(M)^{op}$, and $f\in M$.  Let $\overline{M}$ denote the vector space $M$ with the element $f$ denoted by $\overline{f}$ and  with scalar multiplication defined by $(\lambda,\overline f)\mapsto\lambda\circ\overline{f}= \overline{\overline{\lambda}f}$.   Then
\smallskip
\[
\hbox{
$M$ is a left $E(M)$-module  via 
$
 (A,f)\mapsto A\cdot f=A_1f,
 $
 }
 \]
 \begin{equation}\label{eq:1116231}
 \hbox{
 a right $E(M)^{op}$-module via 
$
 (f,B)\mapsto f\cdot B=B_1f,
 $
}
\end{equation}
\[
\hbox{
  and 
  an $(L_0,R_0)$-bimodule.
  }
  \]
  
  \smallskip
  
 \[\hbox{ 
$\overline{M}$ is a left $E(M)^{op}$-module  via 
$
 (B,\overline{f})\mapsto B\cdot \overline{f}=\overline{B_2f},
 $
 }
 \]
 \begin{equation}\label{eq:1116232}
  \hbox{
   a right $E(M)$-module via
$
 (\overline{f},A)\mapsto \overline{f}\cdot A=\overline{A_2f},
 $
 }
 \end{equation}
 \[
 \hbox{
  and  an $(R_0,L_0)$-bimodule.
  }
\]
  \smallskip

Given an associative triple system  $M$, let 
$
\calA_0=\calA_0(M)=L_0(M)\oplus M\oplus\overline{M}\oplus R_0(M)
$
and write the elements $a$ of $\calA_0$ as matrices
\[
a=\left[\begin{matrix}
A&f\\
\overline{g}&B
\end{matrix}\right]\in \left[\begin{matrix}
L_0(M)&M\\
\overline{M}&R_0(M)
\end{matrix}\right]. 
\]

Define multiplication and involution in $\calA_0$ by
\begin{equation}\label{eq:0906201}
aa'=\left[\begin{matrix}
A&f\\
\overline{g}&B
\end{matrix}\right]
\left[\begin{matrix}
A'&f'\\
\overline{g'}&B'
\end{matrix}\right]
=\left[\begin{matrix}
AA'+\ell(f,g')&A\cdot f'+f\cdot B'\\
\overline{g}\cdot A'+B\cdot\overline{g'}&r(g,f')+B\circ B'
\end{matrix}\right]
\end{equation}
\begin{equation}\label{eq:0906202}
a^\sharp =\left[\begin{matrix}
A&f\\
\overline{g}&B
\end{matrix}\right]^\#=\left[
\begin{matrix}
\overline{A}&g\\
\overline{f}&\overline{B}
\end{matrix}\right].
\end{equation}

Then
$\calA_0(M)$ is an associative *-algebra and for $f,g,h\in M$,
\[
\left[
\begin{matrix}
0&f\\
0&0
\end{matrix}\right]
\left[
\begin{matrix}
0&g\\
0&0
\end{matrix}\right]^\#
\left[
\begin{matrix}
0&h\\
0&0
\end{matrix}\right]
=
\left[
\begin{matrix}
0&[fgh]\\
0&0
\end{matrix}\right].
\]

The map $f\mapsto \left[
\begin{matrix}
0&f\\
0&0
\end{matrix}\right]$ is a triple isomorphism of $M$ into $\calA_0(M)$, the latter considered as an associative triple system with triple product $ab^\#c$, for $a,b,c\in\calA_0(M)$,
We refer to $\calA_0(M)$ as the {\it standard embedding} of $M$.  If the associative triple system $M$ is a  normed space,  and $\|[hgf]\|\le \|f\|\|g\|\|h\|$, then the {\it normed standard embedding} of $M$, denoted by ${\mathcal A}(M)$, is defined in the same way  but  with $R_0(M)$ and $L_0(M)$ replaced by their closures $L(M)$ and $R(M)$ in $B(M\oplus M)$. 

We next recall the construction of the linking algebra of a TRO.  This will motivate our introduction of an anti-C*-algebra.

\begin{remark}\label{rem:0209231}
Let $M\subset B(H,K)$ be a ternary ring of operators (\cite[Subsection 1.3]{PluRusJMAA}), so that $M$ is a C*-ternary ring with the triple product $[xyz]:=xy^*z$.
Then the standard embedding
${\mathcal A}_0(M)$ of $M$ is a pre-C*-algebra, which is *-isomorphic to a dense *-subalgebra of the linking C*-algebra $A_M$ of $M$ (\cite[Example 1.6]{PluRusJMAA}). 
 \end{remark}

Recall that $A_M$ is the closure in $B(K\oplus H)$ of
\[
{\mathcal C}_0:=
\left[\begin{matrix}
MM^*&M\\
M^*&M^*M
\end{matrix}\right]
=\left\{
\left[\begin{matrix}
\sum_ix_iy_i^*&z\\
w^*&\sum_ju_j^*v_j
\end{matrix}\right]
:x_i,y_i,u_j,v_j,z,w\in M
\right\}
\subset \left[\begin{matrix}
B(K)&B(H,K)\\
B(K,H)&B(H)
\end{matrix}\right]
 \]
 with (matrix) multiplication
 \[
 \left[\begin{matrix}
\alpha&z\\
w^*&\beta
\end{matrix}\right]
\times
\left[\begin{matrix}
\alpha'&z'\\
w'^*&\beta'
\end{matrix}\right]
=
\left[\begin{matrix}
\alpha\alpha'+zw'^*&\alpha z'+z\beta'\\
w^*\alpha'+\beta w'^*&w^*z'+\beta\beta'
\end{matrix}\right],
\]
and involution
\begin{equation}\label{eq:0418231}
\left[\begin{matrix}
\alpha&z\\
w^*&\beta
\end{matrix}\right]^*=
\left[\begin{matrix}
\alpha^*&w\\
z^*&\beta^*
\end{matrix}\right],
\end{equation}
and  ${\mathcal C}_0$ is *-isomorphic to ${\mathcal A}_0(M)$ via the map
\begin{equation}\label{eq:0418232}
{\mathcal A}_0(M)\ni 
\left[\begin{matrix}
\ell(x,y)&z\\
\overline{w}&r(u,v)
\end{matrix}\right]\mapsto 
\left[\begin{matrix}
xy^*&z\\
w^*&u^*v
\end{matrix}\right]
\in {\mathcal C}_0.
\end{equation}

We have the following companion result,  which corrects \cite[Example 4.16]{PluRusJMAA}, and is fundamental in this paper.  We omit the straightforward verification.

\begin{proposition}\label{prop:0209231}
Let $M\subset B(H,K)$ be an anti-TRO, that is, as a set, $M$ is equal to a sub-TRO of $B(H,K)$, and it is considered as a C*-ternary ring with the triple product $[xyz]:=-xy^*z$.
Then the standard embedding ${\mathcal A}_0(M)$ is *-isomorphic to the *-algebra 
\[
{\mathcal B}_0=
\left\{
\left[\begin{matrix}
\sum_ix_iy_i^*&z\\
w^*&\sum_ju_j^*v_j
\end{matrix}\right]
:x_i,y_i,u_j,v_j,z,w\in M
\right\}
\subset \left[\begin{matrix}
B(K)&B(H,K)\\
B(K,H)&B(H)
\end{matrix}\right]\subset B(K\oplus H)
 \]
 with multiplication
 \begin{equation}\label{eq:0302231}
 \left[\begin{matrix}
\alpha&z\\
w^*&\beta
\end{matrix}\right]
\cdot
\left[\begin{matrix}
\alpha'&z'\\
w'^*&\beta'
\end{matrix}\right]
=
\left[\begin{matrix}
-\alpha\alpha'+zw'^*&-\alpha z'-z\beta'\\
-w^*\alpha'-\beta w'^*&w^*z'-\beta\beta'
\end{matrix}\right],
\end{equation}
and involution (same as (\ref{eq:0418231}))
\begin{equation}\label{eq:0302232}
\left[\begin{matrix}
\alpha&z\\
w^*&\beta
\end{matrix}\right]^*=
\left[\begin{matrix}
\alpha^*&w\\
z^*&\beta^*
\end{matrix}\right],
\end{equation}
via the map (same as (\ref{eq:0418232}))
\begin{equation}\label{eq:1116233}
{\mathcal A}_0(M)\ni 
\left[\begin{matrix}
\ell(x,y)&z\\
\overline{w}&r(u,v)
\end{matrix}\right]\mapsto 
\left[\begin{matrix}
xy^*&z\\
w^*&u^*v
\end{matrix}\right]
\in {\mathcal B}_0.
\end{equation}

Moreover, the *-isomorphism (\ref{eq:1116233}) extends to a bicontinuous *-isomorphism of ${\mathcal A}(M)$ onto the norm closure 
$\mathcal B$ of ${\mathcal B}_0$ in $B(K\oplus H)$.
\end{proposition}

\begin{definition}\label{def:0422231}
By an {\it anti-C*-algebra} is meant a Banach algebra of the form ${\mathcal A}(M)$, for some anti-TRO $M$ (equivalently,  of the form ${\mathcal A}(M_-)$, for some C*-ternary ring $M$).
\end{definition}

\begin{example}
In Proposition~\ref{prop:0209231}, if $M=\IC$ with triple product $-x\overline{y}z$, then ${\mathcal A}_0(M)$ is equal to $M_2(\IC)$ as a linear space and is an associative *-algebra, with multiplication and involution given by (\ref{eq:0302231}) and (\ref{eq:0302232}) respectively.  This anti-C*-algebra $(M_2(\IC),\cdot)$ has a unit element 
\[
\left[\begin{matrix}
-1&0\\
0&-1
\end{matrix}
\right],
\]
and like its counterpart $(M_2(\IC),\times)$ ($\times$ denoting matrix multiplication), it  has a trivial center and no nonzero proper two-sided ideals. An element 
\[
\left[\begin{matrix}
a&b\\
c&d
\end{matrix}
\right]
\]
is invertible if and only if the ``determinant'' $ad+bc\not=0$.
\end{example}

\begin{remark}
Since the anti-C*-algebra  $(M_2(\IC),\cdot)$ is central and simple, according to Wedderburn's theorem, it is isomorphic to the C*-algebra  $(M_2(\IC), \times)$.  However, this isomorphism cannot be a *-isomorphism, since in that case it would contradict  the last statement in Theorem~\ref{prop:0212231} below.
\end{remark}

The multiplication table for $(M_2(\IC),\cdot)$, with $\{E_{ij}\}$ denoting the usual matrix units,  is
\medskip

\begin{center}
\begin{tabular}{ l || c | c | c | c }
& $E_{11}$ & $E_{12}$ & $E_{21}$ & $E_{22}$ \\ \hline \hline
$E_{11}$ & $-E_{11}$ & $-E_{12}$ & $0$ & $0$ \\ \hline 
$E_{12}$ & $0$ & $0$ & $E_{11}$ & $-E_{12}$ \\ \hline
$E_{21}$ & $-E_{21}$ & $E_{22}$ & $0$ & $0$ \\ \hline
$E_{22}$ & $0$ & $0$ & $-E_{21}$ & $-E_{22}$ 
\end{tabular}
\end{center}
\medskip
and a Wedderburn isomorphism $\phi:(M_2(\IC),\cdot)\rightarrow (M_2(\IC), \times)$, given by 
\[
\phi(E_{ij})=\sum_{p,q}a_{ijpq}E_{pq},
\]
satisfies, for all $p,s,i,j,\ell$, the 32 nonlinear equations with 16 unknowns $a_{ijpq}$
\[
\sum_q a_{ijpq}a_{j\ell qs}=\epsilon(ij\ell)a_{i\ell ps},
\]
where 
\[
\phi(E_{ij})\times \phi(E_{j\ell})=\phi(E_{ij}\cdot E_{j\ell})=\epsilon(ij\ell)E_{i\ell}.
\]

More generally, by Theorem~\ref{thm:0414231} below, any finite dimensional anti-C*-algebra is semisimple and hence isomorphic, but not *-isomorphic to a C*-algebra. It follows that if $M$ is an infinite direct sum of finite dimensional anti-TROs, then ${\mathcal A}(M)$ is an infinite dimensional anti-C*-algebra which is isomorphic, but not *-isomorphic, to a C*-algebra.

Recall from  \cite[Rem.\ 1.4]{PluRusJMAA}  and (\ref{eq:1116231}) that a C*-ternary ring $M$ is a right $R(M)^{op}$-Banach module.
It was proved in \cite[Prop. 2.3(iv)]{PluRusJMAA} that if $M$ is a right $R(M)^{op}$-Hilbert module, then the normed standard embedding ${\mathcal A}(M)$ can be normed to be a C*-algebra.  The following proposition is the converse of the latter result.

\begin{proposition}\label{prop:0410231}
If $M$ is a C*-ternary ring, then the normed standard embedding ${\mathcal A}(M)$  can be normed to be a C*-algebra if and only if $M$ is a right $R(M)^{op}$-Hilbert module.
\end{proposition}
\begin{proof}
As noted above, one direction has been proved in \cite[Prop. 2.3(iv)]{PluRusJMAA}. 
Suppose then that ${\mathcal A}(M)$ can be normed to be a C*-algebra. Since $M$ is an associative triple subsystem of the supposed C*-algebra ${\mathcal A}(M)$ it is therefore isomorphic as a C*-ternary ring to a TRO, say $V\subset B(H)$.  By \cite[Prop. 2.6 and Example 1.6]{PluRusJMAA}, if $\phi:M\rightarrow V$ is the isomorphism, then the map
\[
{\mathcal A}_0(M)\ni
\left[\begin{matrix}
([gh\cdot],[hg\cdot])&f\\
\overline{g}&([\cdot hk],[\cdot kh])
\end{matrix}\right]
\mapsto 
\left[\begin{matrix}
\phi(g)\phi(h)^*&\phi(f)\\
\phi(g)^*&\phi(h)^*\phi(k)
\end{matrix}\right]\in A_V
\]
extends to a *-isomorphism of ${\mathcal A}(M)$ onto the linking algebra 
\[
A_V=\left[\begin{matrix}
\overline{VV^*}&V\\
V^*&\overline{V^*V}
\end{matrix}\right]
\] 
of $V$. 

Let us denote the C*-algebra $\overline{V^*V}$ by $\mathcal D$.  The TRO $V$ is a right $\mathcal D$-Hilbert module with inner product $\langle u,v\rangle_{\mathcal D} =v^*u$. With $\gamma$ denoting the *-isomorphism of $R(M)$ onto $\mathcal D$, namely, $\gamma: [\cdot hk],[\cdot kh])\mapsto \phi(h)^*\phi(k)$, the map $\langle f,g\rangle_{R(M)}:=\gamma^{-1}(\langle \phi(f),\phi(g)\rangle_{\mathcal  D})$ is an $R(M)$-valued inner product making $M$ into a right $R(M)^{op}$-Hilbert module. Indeed, it is obvious that $\langle f,f\rangle_{R(M)}\ge 0$, $\langle f,g\rangle_{R(M)}^*=\langle g,f\rangle_{R(M)}$, and 
\[
\|\langle f,f\rangle_{R(M)}\|=\|\langle \phi(f),\phi(f)\rangle_{\mathcal D}\|=\|\phi(f)\|^2=\|f\|^2.
\]
It remains to show that\footnote{Recall that $ \circ$ denotes the product in $R(M)^{op}$.}
$
\langle f\cdot r(h,k),g\rangle_{R(M)}=\langle f,g\rangle_{R(M)}\circ r(h,k).
$
First,
\[
\langle f\cdot r(h,k),g\rangle_{R(M)}=\langle[fhk],g\rangle_{R(M)}=\gamma^{-1}\langle\phi[fhk],\phi(g)\rangle_{\mathcal D}=\gamma^{-1}(\phi(g)^*\phi[fhk])=([\cdot g[fhk]],[\cdot[fhk]g])
,\]
and second, using \cite[Lemma 1.1(iii)]{PluRusJMAA},
\begin{eqnarray*}
\langle f,g\rangle_{R(M)}\circ r(h,k)&=&\gamma^{-1}(\langle \phi(f),\phi(g)\rangle_{\mathcal D})\circ ([\cdot hk],[\cdot kh])\\
&=&\gamma^{-1}(\phi(g)^*\phi(f))\circ ([\cdot hk],[\cdot kh])\\
&=&([\cdot gf],[\cdot fg])\circ   ([\cdot hk],[\cdot kh])\\
&=&([\cdot hk],[\cdot kh])   ([\cdot gf],[\cdot fg])\\
&=&([\cdot hk][\cdot gf],[\cdot fg][\cdot kh])\\
&=&([[\cdot gf]hk],[[\cdot kh]fg],
\end{eqnarray*}
as desired.
\end{proof}

In general, the normed standard embedding ${\mathcal A}(M)$ of a C*-ternary ring is a *-algebra which can be normed to be a C*-algebra if $M$ is an $R^{op}$-Hilbert module by using a *-isomorphism  $\pi$ into $B(M\oplus R)$, with $M\oplus R$ considered as a Hilbert $R^{op}$-module (\cite[Section 6]{PluRusJMAA}). If $M$ is not a Hilbert $R^{op}$-module, then $M\oplus R$ is just a Banach space (under any convenient $\ell^p_2$-norm).  The proof that $\pi$ is a homomorphism in this case is the same as the proof in the Hilbert module case in \cite[pp.\ 34-35]{PluRusJMAA}.  The proof given in \cite[p.\ 36]{PluRusJMAA} for completeness works in this case also.

 In the Hilbert module case in \cite[Section 6]{PluRusJMAA}, $\pi$ was shown to be injective by using the fact that $\pi$ was a self adjoint mapping, which however does not make sense if $M\oplus R$ is not a Hilbert module over $R^{op}$. However, a modification of that proof will now be given in the proof of the following theorem, which is the replacement for  \cite[Prop. 2.7]{PluRusJMAA}, to show that in this case, $\pi$ is injective.

\begin{theorem}\label{prop:0212231}
The normed standard embedding   ${\mathcal A}(M)$ of a C*-ternary ring $M$ is the direct sum of a C*-algebra and an anti-C*-algebra ${\mathcal B}$.  The anti-C*-algebra ${\mathcal B}$ is a semisimple Banach algebra  with a continuous involution and a bounded approximate identity, which cannot be renormed to be a C*-algebra. It is, however, linearly homeomorphic to a C*-algebra.
\end{theorem}
\begin{proof}
By Zettl's decomposition (\cite[Theorem 2.1]{PluRusJMAA}), $M=M_+\oplus M_-$, where $M_+$ is isomorphic as a C*-ternary ring to a TRO, and $M_-$ is isomorphic as a C*-ternary ring to an anti-TRO. By Remark~\ref{rem:0209231} and \cite[Prop. 2.6]{PluRusJMAA}, ${\mathcal A}(M_+)$ is *-isomorphic to a C*-algebra. Moreover ${\mathcal A}(M)={\mathcal A}(M_+)\oplus {\mathcal A}(M_-)$. 
It remains to show that ${\mathcal B}={\mathcal A}(M_-)$ has the required properties.

 Recall that 
the map $\pi:\calA(M)\rightarrow B(M\oplus R)$ to the bounded operators on the Banach space $M\oplus R$, normed as above,  say for definiteness,
\[
\left\|\left[\begin{matrix}
f'\\
B'
\end{matrix}\right]\right\|=\left( \|f'\|^2+\|B'\|^2
\right)^{1/2},
\]
is defined, for $a=\left[\begin{matrix}
A&f\\
\overline{g}&B
\end{matrix}\right]\in \calA(M)$ by 
\begin{equation}\label{eq:0818201}
\pi(a)\left[\begin{matrix}
f'\\
B'
\end{matrix}\right] =
\left[\begin{matrix}
A&f\\
\overline{g}&B
\end{matrix}\right]\left[\begin{matrix}
f'\\
B'
\end{matrix}\right]
=\left[\begin{matrix}
A\cdot f'+f \cdot B'\\
r(g,f')+B\circ B'
\end{matrix}\right].
\end{equation}

We have,
\[
\|\pi(a)\|=
\left\|\left[\begin{matrix}
A&f\\
\overline{g}&B
\end{matrix}\right] \right\|=\sup_{\|(f',B')\|\le 1}\left\|\left[\begin{matrix}
A\cdot f'+f \cdot B'\\
r(g,f')+B\circ B'
\end{matrix}\right]\right\|\\,
\]
so that 

\[
\left\|\left[\begin{matrix}
A&f\\
\overline{g}&B
\end{matrix}\right] \right\|\ge \sup_{\|f'\|\le 1}
(\|A\cdot f'\|^2+\|r(g,f')\|^2)^{1/2}\ge\sup_{\|f'\|\le 1}\|A_1f'\|=
 \|A_1\|=\|A\|,
\]
and 

\[
\left\|\left[\begin{matrix}
A&f\\
\overline{g}&B
\end{matrix}\right] \right\|\ge \sup_{\|f'\|\le 1}
(\|A\cdot f'\|^2+\|r(g,f')\|^2)^{1/2}\ge\sup_{\|f'\|\le 1} \|r(g,f')\|=\sup_{\|f'\|\le 1}\sup_{\|f''\|\le 1} \|[f''f'g]\|=\|g\|,
\]
the last equality  holding by applying \cite[Lemma 2.3(iv)]{Hamana99} to $M_+$ and $M_-$.
Similarly, 
\[
\left\|\left[\begin{matrix}
A&f\\
\overline{g}&B
\end{matrix}\right] \right\|\ge \|B\|\hbox{ and }
\left\|\left[\begin{matrix}
A&f\\
\overline{g}&B
\end{matrix}\right] \right\|\ge \|f\|.
\]

Thus, if 
\[\left[\begin{matrix}
A&f\\
\overline{g}&B
\end{matrix}\right] \rightarrow 0, \hbox{ then } \left[\begin{matrix}
A&f\\
\overline{g}&B
\end{matrix}\right]^\#=\left[\begin{matrix}
\overline{A}&g\\
\overline{f}&\overline{B}
\end{matrix}\right] \rightarrow 0.
\]

Let us now show that $\pi$ is injective.  
For  $a=\left[\begin{matrix}
A&f\\
\overline{g}&B
\end{matrix}\right]\in \calA$, if $\pi(a)=0$, then  by (\ref{eq:0818201})
\[
A\cdot f'+f \cdot B'=0\hbox{ and }
r(g,f')+B\circ B'=0
\] 
 for all $f'\in M, B'\in R$, and in particular,
\begin{equation}\label{eq:0818209}
A\cdot f'=0\hbox{ and } f\cdot B'=0,
\end{equation}
and
\begin{equation}\label{eq:08182010}
r(g,f')=0\hbox{ and } B\circ B'=0.
\end{equation}

From (\ref{eq:0818209}) with $B'=r(f,f)$, $[fff]=0$ so $f=0$.
From (\ref{eq:08182010}), $\overline{B}B=0$  and $r(g,g)=0$, so $B=0$ and $g=0$.
 From (\ref{eq:0818209}) with  $A=\ell(g,h)$,
$[ghf']=0$  so that $A_1=L(g,h)=0$, and $L(h,g)=L(g,h)^*=0$ so $A=(L(g,h),L(h,g))=0$.
By the same argument, if $A=\sum_i\ell(g_i,h_i)$, then $A=0$.
Now suppose $A\in L$, let $\epsilon>0$ and choose $A'=\sum_i\ell(g_i,h_i)$ with $\|A-A'\|<\epsilon$. Then $\|A'\cdot f'\|=\|(A-A')\cdot f'\|\le\epsilon\|f'\|$, so that $\|A\|\le \|A-A'\|+\|A'\|<2\epsilon$, and $A=0$.

Next, if $(A_\gamma)_{\gamma\in\Gamma}$ is an approximate identity for the C*-algebra $L(M)$, and $(B_\delta)_{\delta\in \Delta}$ is an approximate identity for the C*-algebra $R(M)$, then 
\[
\left(
\left[\begin{matrix}
A_\gamma&0\\
0&B_\delta
\end{matrix}\right]\right)_{(\gamma,\delta)\in \Gamma\times\Delta}
\]
is an approximate identity for ${\mathcal A}(M)$.  To see this, one needs to prove that $A_\gamma\cdot f\rightarrow f$ and $B_\gamma\cdot g\rightarrow g$. For example, if $A_\gamma=(A_\gamma^1,A_\gamma^2)$ 
and%
\footnote{Cube roots exist in C*-ternary rings. It suffices to know that cube roots exist in TROs and in anti-TROs.  Since a TRO can be made into a JB*-triple by symmetrizing its triple product, TROs have cube roots (\cite[p.\ 1135]{Peralta}). Let $M$ be an anti-TRO, $M\subset B(H,K)$, with triple product $[xyz]=-xy^*z$, and if $a\in M$, then $-a\in M$ so $-a=bb^*b$ for some $b\in M$ and $a=(-b)(-b^*)(-b)=-bb^*b=[bbb]$.}
$f=[ggg]$, then $A_\gamma\cdot f=A_\gamma^1f=A_\gamma^1L(g,g)g
\rightarrow L(gg)g=f$.

We now have that ${\mathcal A}(M_-)$ is a Banach algebra with a continuous involution.
Since a C*-ternary ring $M$ is embedded in its standard embedding  ${\mathcal A}(M)$ as a sub-associative triple system,  it is not possible for ${\mathcal A}(M_-)$ to support a norm making it a C*-algebra, since in that case $M_-$ would be isomorphic as a C*-ternary ring to a sub-TRO of the supposed  C*-algebra ${\mathcal A}(M_-)$, which violates the uniqueness of the Zettl decomposition.

For the proof of semisimplicity, see Theorem~\ref{thm:0414231}.
For the proof of the last statement, see Proposition~\ref{prop:1117231}.
\end{proof}

We close this section by summarizing some results already mentioned which characterize TROs in the context of C*-ternary rings. The equivalence of (i) (or (ii)) and (v) in \cite{NeaRus03} answered a question posed by Zettl in \cite[p.\ 136]{Zettl83}, namely, ``characterizing the C*-ternary rings which yield $T=I$'', equivalently, that are isomorphic to a TRO.  Conditions (iii) and (iv) provide two new answers to Zettl's question.

\begin{theorem}\label{thm:0413231}
 Let $M$ be a C*-ternary ring. The following are equivalent:
\begin{description}
\item[(i)] $M$ is isomorphic as a C*-ternary ring to a TRO.
\item[(ii)] $M_-=0$.
\item[(iii)] $M$ is a right $R(M)^{op}$-Hilbert module.
\item[(iv)] ${\mathcal A}(M)$ can be normed to be a C*-algebra.
\item[(v)] $M$ is a JB*-triple under the triple product 
\[
\{abc\}=\frac{[abc]+[bca]}{2}.
\]
\end{description}
\end{theorem}
\begin{proof}\noindent
\begin{description}
\item[(i)$\Leftrightarrow$ (v)]: \cite[p.\ 342]{NeaRus03}
\item[(i)$\Leftrightarrow$ (ii)]: Zettl's representation theorem (\cite[Theorem 2.1]{PluRusJMAA})
\item[(iii)$\Rightarrow$ (iv)]: \cite[Prop. 2.3 (iv)]{PluRusJMAA}
\item[(iv)$\Rightarrow$ (iii)]: Proposition~\ref{prop:0410231}
\item[(i)$\Rightarrow$ (iv)]:  \cite[Lemma 2.6, Example 1.6]{PluRusJMAA}
\item[(iv)$\Rightarrow$ (i)]:  see the last  paragraph of the proof of Theorem~\ref{prop:0212231}.
\end{description}
\end{proof}

\section{Semisimplicity}\label{sec:0410232}

In this section, following closely  \cite{Meyberg72},  we prove that an anti-C*-algebra is semisimple. Along the way we prove that TROs and anti-TROs are semisimple.

\begin{definition}
A {\it left ideal} (resp.\ {\it right ideal})  $I$ in an associative triple system $M$ is a linear subspace which satisfies $[MMI]\subset I$ (resp.\ $[IMM]\subset I$). An {\it ideal} is both a left ideal and a right ideal which also satisfies $[MIM]\subset I$. (In a C*-ternary ring, a subspace which is both a left ideal and a right ideal automatically satisfies
$[MIM]\subset I$.)
\end{definition}

With $M$ a C*-ternary ring, we denote by $ {\mathcal A}_1(M)$ the unitization of ${\mathcal A}(M)$:
\begin{equation}\label{eq:0315231}
 {\mathcal A}_1(M)= {\mathcal A}(M)\oplus \IC\hbox{ Id}_{E(M)\oplus E(M)^{op}}
=
\left[\begin{matrix}
L(M)+\IC E_1&M\\
\overline{M}&R(M)+\IC E_2
\end{matrix}\right],
\end{equation}
where $E_1=\hbox{ Id}_{E(M)}$ and $ E_2=\hbox{ Id}_{E(M)^{op}}$. The Peirce components  of  ${\mathcal A}_1(M)$ relative to the idempotent $E_1$ are therefore  precisely the entries in the matrix (\ref{eq:0315231}). With ${\mathcal A}_1$ denoting ${\mathcal A}_1(M)$,
\begin{eqnarray*}
{\mathcal A}_1&=&
({\mathcal A}_1)_{11}\oplus({\mathcal A}_1)_{10}\oplus({\mathcal A}_1)_{01}\oplus({\mathcal A}_1)_{00}\\
&=&E_1{\mathcal A}_1 E_1\oplus E_1 {\mathcal A}_1 E_2\oplus E_2 {\mathcal A}_1 E_1\oplus E_2 {\mathcal A}_1 E_2\\
&=&(L(M)+\IC E_1)\oplus M\oplus \overline{M}\oplus (R(M)+\IC E_2).
\end{eqnarray*}
For any closed ideal $I$ in ${\mathcal A}_1$, we have (cf.\ \cite[Lemma 7, p. 18]{Meyberg72})
\[
I=\oplus_{i,j\in \{0,1\}}(({{\mathcal A}_1})_{ij}\cap I).
\]
Thus, if $I$ is a closed ideal in ${\mathcal A}(M)\subset {\mathcal A}_1(M)$, then
\begin{equation}\label{eq:0317231}
I=(I\cap L(M))\oplus (I\cap M)\oplus (I\cap \overline{M})\oplus (I\cap R(M)),
\end{equation}
and $I\cap M$ is a closed ideal in $M$ (cf. \cite[Lemma 4, p.\ 31]{Meyberg72}). Indeed, if $b\in  I\cap M$ and $x,y\in M$, then 
\[
\left[\begin{matrix}
0&[xyb]\\
0&0
\end{matrix}\right]\
=\left[\begin{matrix}
0&x\\
0&0
\end{matrix}\right]
\left[\begin{matrix}
0&y\\
0&0
\end{matrix}\right]^\#
\left[\begin{matrix}
0&b\\
0&0
\end{matrix}\right]\in I,
\]
and similarly $[bxy]\in I\cap M$, so $I\cap M$ is an ideal in $M$.
 
 Moreover, if $I\subset L(M)\oplus R(M)$, then $I=0$.  Indeed, since $I\cap M=I\cap\overline{M}=0$, $I=(L(M)\cap I)\oplus(R(M)\cap I)$. Let $A=(A_1,A_2)\in L(M)\cap I$ and $f\in M$. Then
\begin{equation}\label{eq:0318231}
\left[\begin{matrix}
A&0\\
0&0
\end{matrix}\right]
\left[\begin{matrix}
0&f\\
0&0
\end{matrix}\right]=
\left[\begin{matrix}
0&A\cdot f\\
0&0\end{matrix}\right]=
\left[\begin{matrix}
0&A_1f\\
0&0\end{matrix}\right]\in I\cap M=0,
\end{equation}
so $A_1=0$ and $A_2=A_1^*=0$, and $A=0$.  Similarly $R(M)\cap I=0$.
\smallskip

Recall that the Jacobson radical of an associative algebra $A$ is the ideal consisting of the set of elements $x\in A$ which are quasi-invertible in every homotope $A_u$ of $A$, that is, for every $u\in A$, there exists $y\in A$  (depending on $x$ and $u$) such that $y-x=xuy=yux$.

\begin{definition}
The  (Jacobson) {\it radical} Rad $M$ of an associative triple system, such as a C*-ternary ring $M$,  is the set of elements $x \in M$ which are quasi-invertible in every homotope $M_u$ of M, that is,  for every $u\in M$, there exists $y\in M$ such that
\[
y-x=[yux]=[xuy].
\]
A C*-ternary ring is said to be {\it semisimple} if its radical is 0.
\end{definition}

\begin{lemma}\label{lem:0317231}(cf.\  \cite[Lemma 5, p.\ 32]{Meyberg72}) 
For $x$ and $u$ in a C*-ternary ring $M$, the following are equivalent:
\begin{description}
\item[(i)] $x$ is quasi-invertible in $M_u$, that is, there exists $y\in M$ such that $y-x=[yux]=[xuy]$.
\item[(ii)] $\left[\begin{matrix}
0&x\\
0&0
\end{matrix}\right]$ is quasi-invertible in  ${\mathcal A}(M)_{\left[\begin{matrix}
0&u\\
0&0
\end{matrix}\right]^\#}$, that is, there exists $Y\in {\mathcal A}(M)$ such that 
\[
Y-\left[\begin{matrix}
0&x\\
0&0
\end{matrix}\right]
=Y \left[\begin{matrix}
0&u\\
0&0
\end{matrix}\right]^\#
\left[\begin{matrix}
0&x\\
0&0
\end{matrix}\right]=
\left[\begin{matrix}
0&x\\
0&0
\end{matrix}\right]
\left[\begin{matrix}
0&u\\
0&0
\end{matrix}\right]^\#
Y.
\]
\end{description}
\end{lemma}
\begin{proof}
Assume that (i) holds.  Then just check that
\[
\left[\begin{matrix}
0&y\\
0&0
\end{matrix}\right]
-
\left[\begin{matrix}
0&x\\
0&0
\end{matrix}\right]
=
\left[\begin{matrix}
0&y\\
0&0
\end{matrix}\right] \left[\begin{matrix}
0&u\\
0&0
\end{matrix}\right]^\#
\left[\begin{matrix}
0&x\\
0&0
\end{matrix}\right]=
\left[\begin{matrix}
0&x\\
0&0
\end{matrix}\right]
\left[\begin{matrix}
0&u\\
0&0
\end{matrix}\right]^\#
\left[\begin{matrix}
0&y\\
0&0
\end{matrix}\right].
\]
Conversely, assume (ii) holds.  With $Y=
\left[\begin{matrix}
A&f\\
\overline{g}&B
\end{matrix}\right],
$ we have
\[
\left[\begin{matrix}
A&f\\
\overline{g}&B
\end{matrix}\right]
-
\left[\begin{matrix}
0&x\\
0&0
\end{matrix}\right]
=
\left[\begin{matrix}
A&f\\
\overline{g}&B
\end{matrix}\right]
\left[\begin{matrix}
0&0\\
\overline{u}&0
\end{matrix}\right]
\left[\begin{matrix}
0&x\\
0&0
\end{matrix}\right]=
\left[\begin{matrix}
0&x\\
0&0
\end{matrix}\right]
\left[\begin{matrix}
0&0\\
\overline{u}&0
\end{matrix}\right]
\left[\begin{matrix}
A&f\\
\overline{g}&B
\end{matrix}\right],
\]
which reduces to 
\[
\left[\begin{matrix}
A&f-x\\
\overline{g}&B
\end{matrix}\right]
=
\left[\begin{matrix}
0&\ell(f,u)\cdot x\\
0&r(B\cdot\overline{u},x)
\end{matrix}\right]
=
\left[\begin{matrix}
\ell(x,\overline{u}\cdot A)&x\cdot r(u,f)\\
0&0
\end{matrix}\right].
\]
Thus, $A=0,g=0,B=0$ and 
$f-x=[fux]=[xuf]$.
\end{proof}

 With $M$ a C*-ternary ring, and since $\hbox{Rad }  {\mathcal A}$ is an ideal,  then by (\ref{eq:0317231})
 (cf.\ \cite[Theorem 5, p.\ 18]{Meyberg72}),
   \[
\hbox{Rad }  {\mathcal A}=[(\hbox{Rad } {\mathcal A})\cap L(M)]\oplus [(\hbox{Rad } {\mathcal A})\cap M]\oplus ([\hbox{Rad } {\mathcal A}) \cap \overline{M}]\oplus [(\hbox{Rad } {\mathcal A})\cap R(M)].
\]

Moreover
 
\begin{equation}\label{eq:0317234}
(\hbox{Rad } {\mathcal A})\cap L(M)=\hbox{Rad } L(M)
\end{equation}
and
\begin{equation}\label{eq:0317235}
(\hbox{Rad } {\mathcal A})\cap R(M)=\hbox{Rad } R(M).
\end{equation}
 To prove (\ref{eq:0317234}), let $X= \left[\begin{matrix}
A&0\\
0&0
\end{matrix}\right]\in (\hbox{Rad } {\mathcal A})\cap L(M)$. Then for all $U= \left[\begin{matrix}
A''&f''\\
\overline{g''}&B''
\end{matrix}\right]\in {\mathcal A}$, there exists $Q= \left[\begin{matrix}
A'&f'\\
\overline{g'}&B'
\end{matrix}\right]\in {\mathcal A}$ with $Q- X=QUX=XUQ$, which after calculation yields $$A'-A=A'A''A=AA''A',$$
so 
$A\in \hbox{Rad }L(M)$ and $(\hbox{Rad } {\mathcal A})\cap L(M)\subset \hbox{Rad } L(M)$ .

To show the reverse 
inclusion\footnote{Since $L(M)$ and $R(M)$ are C*-algebras, their radicals vanish; the proof given here for the reverse inclusion is therefore valid for general associative triple systems.} in (\ref{eq:0317234}),
we use the following two lemmas.

\begin{lemma}\label{lem:0415231}
(cf.\ \cite[Lemma 4, p.\ 12]{Meyberg72} ``Symmetry principle'')
Given $x,y$ in an associative algebra $A$,
$x$ is quasi-invertible in $A_y$ if and only if $y$ is quasi-invertible in $A_x$. That is,
\begin{equation}\label{eq:0414231}
\exists u\in A, u-x=uyx=xyu
\end{equation}
if and only if 
\begin{equation}\label{eq:0414232}
\exists w\in A, w-y=wxy=xyw.
\end{equation}
\end{lemma}
\begin{proof}
From (\ref{eq:0414231}), we have
\[
uy-xy=uyxy=xyuy,
\]
which means that $xy$ is quasi-invertible in $A$.
With $w=uy$
\[
w-xy=wxy=xyw
\]
and 
\[
yw-yxy=ywxy=yxyw
\]
so that 
\[
(yw+y)-y=yw=ywxy+yxy=yxyw+yxy
\]
and
\[
(yw+y)-y=(yw+y)xy=yx(yw+y)
\]
and therefor $y$ is quasi-invertible in $A_x$.
The converse follows by interchanging $x$ and $y$.
\end{proof}

\begin{lemma}\label{lem:0317232}
(cf.\ \cite[Lemma 5, p.\ 13]{Meyberg72} ``Shifting principle'')
Let $\varphi$ and $\psi$ be endomorphisms of an associative algebra $A$ which satisfy, for all $x,y,z\in A$,
\begin{equation}\label{eq:0415231}
\varphi(x)z\varphi(y)=\varphi(x\psi(z)y),
\end{equation}
and
\begin{equation}\label{eq:0415232}
\psi(x)z\psi(y)=\psi(x\varphi(z)y).
\end{equation}
Then  for every $x\in A$, 
$x\hbox{ is quasi-invertible in }A_{\psi(y)}$ if and only if $\varphi(x)\hbox{ is quasi-invertible in } A_y$, that is, 
\begin{equation}\label{eq:0317232}
\exists u\in A,
u-x=x\psi(y)u=u\psi(y)x,
\end{equation}
if and only if
\begin{equation}\label{eq:0317233}
\exists w\in A, 
w-\varphi(x)=\varphi(x)yw=wy\varphi(x).
\end{equation}
In each case, $\varphi(u)=w$.
\end{lemma}
\begin{proof}
Applying $\varphi$ to (\ref{eq:0317232}) using (\ref{eq:0415231}) yields
\[
\varphi(u)-\varphi(x)=\varphi(x\psi(y)u)=
\varphi(x)y\varphi(u)=\varphi(u\psi(y)x)
=\varphi(u)y\varphi(x),
\]
so that $\varphi(x)$ is quasi-invertible in $A_y$ with quasi inverse $\varphi(u)$, proving one direction as well as the last statement.

Conversely, assuming that $\varphi(x)$ is quasi-invertible in $A_y$, then by Lemma~\ref{lem:0415231}, $y$ is quasi-invertible in $A_{\varphi(x)}$, so there exists $u\in A$ with 
\[
u-y=u\varphi(x)y=y\varphi(x)u.
\]
Applying $\psi$  using (\ref{eq:0415232}) yields
\[
\psi(u)-\psi(y)=\psi(u)x\psi(y)=\psi(y)x\psi(u),
\]
so that $\psi(y)$ is quasi-invertible in $A_x$, and by Lemma~\ref{lem:0415231}, $x$ is quasi-invertible in $A_{\psi(y)}$.
\end{proof}

We can now complete the proof of (\ref{eq:0317234}).
We shall apply Lemma~\ref{lem:0317232} with $\varphi=\psi=E_1\cdot E_1$ in the case that $x= \left[\begin{matrix}
A&0\\
0&0
\end{matrix}\right]=E_1xE_1\in \hbox{Rad }L(M)$.
Then $x$ is quasi-invertible in ${\mathcal A}_{\psi(y)}$ for every $y\in {\mathcal A}$ and therefore $x=\varphi (x)$ is quasi-invertible in $\mathcal A_y$ for every $y\in {\mathcal A}$, so $x\in \hbox{Rad }({\mathcal A})$. This proves (\ref{eq:0317234}), and (\ref{eq:0317235}) follows by a parallel argument.

We have proved part of the following theorem (cf.\ \cite[Theorem 3, p.\ 32]{Meyberg72}).

\begin{theorem}\label{thm:0406231}
If $M$ is a C*-ternary ring, and ${\mathcal A}(M)$ its normed standard embedding, then
\begin{description}
\item[(i)]
 $
\hbox{Rad }  {\mathcal A}=[\hbox{Rad }  L(M)]\oplus [\hbox{Rad } M]\oplus [\overline{\hbox{Rad } M}]\oplus [(\hbox{Rad }R(M)]
$
\item[(ii)] $\hbox{Rad }M$ is an ideal in $M$
\item[(iii)] ${\mathcal A}(M)$ is semisimple if and only if $M$ is semisimple.
\end{description}
\end{theorem}
\begin{proof}
To finish the proof of (i), it suffices to show that $\hbox{Rad }M=[\hbox{Rad }{\mathcal A}(M)]\cap M$ and $\hbox{Rad }\overline{M}=[\hbox{Rad }{\mathcal A}(M)]\cap \overline{M}$. Suppose that $x\in[\hbox{Rad }{\mathcal A}(M)]\cap M$, that is, $ \left[\begin{matrix}
0&x\\
0&0
\end{matrix}\right]$ is quasi-invertible in ${\mathcal A}(M)_Y$ for every $Y=\left[\begin{matrix}
A&f\\
\overline{g}&B
\end{matrix}\right]\in {\mathcal A}(M)$. Thus, there exists $Q=\left[\begin{matrix}
A'&f'\\
\overline{g'}&B'
\end{matrix}\right]\in{\mathcal A}(M)$ such that
\[
Q- \left[\begin{matrix}
0&x\\
0&0
\end{matrix}\right]=
 \left[\begin{matrix}
0&x\\
0&0
\end{matrix}\right]YQ=QY \left[\begin{matrix}
0&x\\
0&0
\end{matrix}\right],
\]
which leads to $A'=0,B'=0,g'=0$ and 
$
f'-x=[xgf']=[f'gx],
$
and this means that $x\in \hbox{Rad }M$, that is,
$[\hbox{Rad }{\mathcal A}(M)]\cap M\subset\hbox{Rad }M$.

Conversely, let $x\in \hbox{Rad }M$, so that $x$ is quasi-invertible in $M_u$ for all $u\in M$. Then by Lemma~\ref{lem:0317231},  $ \left[\begin{matrix}
0&x\\
0&0
\end{matrix}\right]$ is quasi-invertible in  ${\mathcal A}_{ \left[\begin{matrix}
0&0\\
\overline{u}&0
\end{matrix}\right]}$.  Define
$\varphi,\psi:{\mathcal A}(M)\rightarrow {\mathcal A}(M)$ by
\[
\varphi\left( \left[\begin{matrix}
A&f\\
\overline{g}&B
\end{matrix}\right]\right)=E_1\left[\begin{matrix}
A&f\\
\overline{g}&B
\end{matrix}\right]E_2=\left[\begin{matrix}
0&f\\
0&0
\end{matrix}\right]
\]
and
\[
\psi\left( \left[\begin{matrix}
A&f\\
\overline{g}&B
\end{matrix}\right]\right)=E_2\left[\begin{matrix}
A&f\\
\overline{g}&B
\end{matrix}\right]E_1=\left[\begin{matrix}
0&0\\
\overline{g}&0
\end{matrix}\right].
\]
Then by Lemma~\ref{lem:0317232},
$\varphi\left( \left[\begin{matrix}
0&x\\
0&0
\end{matrix}\right]\right)=\left[\begin{matrix}
0&x\\
0&0
\end{matrix}\right]$ is quasi-invertible in ${\mathcal A}(M)_{\psi(y)}$, that is, 
$\left[\begin{matrix}
0&x\\
0&0
\end{matrix}\right]\in [\hbox{Rad }{\mathcal A}(M)]\cap M$, proving that $\hbox{Rad }M=[\hbox{Rad }{\mathcal A}(M)]\cap M$.

Since $\hbox{Rad }{\mathcal A}$ is self-adjoint ($X$ is quasi-invertible in ${\mathcal A}_U$ for all $U\in {\mathcal A}$ if and only if $X^\#$ is quasi-invertible in ${\mathcal A}_{U^\#}$ for all $U^\#\in {\mathcal A}$),  we have

$\left[\begin{matrix}
0&0\\
\overline{u}&0
\end{matrix}\right]\in (\hbox{Rad }{\mathcal A})\cap \overline{M}$ if and only if 
$\left[\begin{matrix}
0&u\\
0&0
\end{matrix}\right]\in [ (\hbox{Rad }{\mathcal A})\cap \overline{M}]^\#=
(\hbox{Rad }{\mathcal A})\cap M=\hbox{Rad }M$. Therefore $\overline{\hbox{Rad }M}=(\hbox{Rad }{\mathcal A})\cap \overline{M}$, completing the proof of (i), and also proving (ii) as in (\ref{eq:0317231}).

Finally, if $\hbox{Rad }{\mathcal A}=0$, obviously, $\hbox{Rad }M=0$.  Conversely, if $\hbox{Rad }M=0$, then
$\hbox{Rad }{\mathcal A} \subset L(M)\oplus R(M)$ so as in (\ref{eq:0318231}),  $\hbox{Rad }{\mathcal A}=0$.
\end{proof}

Since C*-algebras are semisimple, we have

\begin{corollary}\label{cor:0406231}
If $M$ is a TRO, then $M$ is semisimple.
\end{corollary}

\begin{proposition}\label{prop:0406231}
Let $M\subset B(H,K)$ be an anti-TRO, that is, as a linear space, $M$ is equal to a sub-TRO $M'$ of $B(H,K)$, and it is considered as a C*-ternary ring with the triple product $[xyz]:=-xy^*z$.  Then $\hbox{Rad }M=\hbox{Rad }M'$. Thus, an anti-TRO is semisimple.
\end{proposition}
\begin{proof}
We note first that ideals in $M$ are also ideals in  $M'$, and conversely.
According to \cite[Theorem 9]{Myung75}, for any associative triple system $N$  of the second kind, with ternary product denoted $[xyz]$,
\begin{equation}\label{eq:0406233}
\hbox{Rad }N=\{ x\in N:\hbox{ the principal ideal }[xNN]\hbox{ or }[NNx]\hbox{ is quasi-regular in }N  \}.
\end{equation}
Quasi-regular for the right ideal $[xNN]$, which is equivalent to quasi-regularity for left ideals \cite[p.\ 32]{Myung75}, means that 
\begin{equation}\label{eq:0419231}
N=\{u-[vyu]: u,y\in N, v\in [xNN]\},
\end{equation}
so in our case, we need to prove for the anti-TRO $M$, that
\[
M=\{u+vy^*u: u,y\in M, v\in xM^*M\}.
\]

By (\ref{eq:0406233}), for the TRO $M'$,
\[
\hbox{Rad }M'=\{ x\in M':\hbox{ the principal ideal }xM'^*M'\hbox{ or }M'M'^*x\hbox{ is quasi-regular in }M'  \},
\]
that is,
$$
M'=\{u-vy^*u: u,y\in M', v\in xM'^*M'\}.
$$

In addition to ideals of $M$ and $M'$ coinciding,  quasi-regularity is preserved, that is,
\[
\{u+vy^*u: u,y\in M, v\in xM^*M\}=\{u-v(-y)^*u: u,-y\in M, v\in xM^*M\}=M'=M,
\]
proving (\ref{eq:0419231}) with $N=M'$.
Hence $\hbox{Rad }M=\hbox{Rad }M'$.
\end{proof}

\begin{theorem}\label{thm:0414231}
The Banach algebra $\mathcal B$ in Theorem~\ref{prop:0212231} is semisimple. Hence, all C*-ternary rings and all anti-C*-algebras are semisimple.
\end{theorem}
\begin{proof} This follows from Theorem~\ref{thm:0406231}(iii) and Proposition~\ref{prop:0406231}.
\end{proof}

\section{Ideals}\label{sec:0410233}

There is a one to one correspondence between closed ideals in a TRO $M$ and  closed ideals in the C*-algebra $R(M)$. This was proved directly (that is, not using \cite{RaeWil98}) in \cite[Prop.\ 3.8]{BunTimQJM}.  This result was extended to C*-ternary rings in 
\cite[Prop.\ 4.2]{Abadie2017}, by using the TRO $M_+\oplus M_-^{op}$ and appealing to \cite[Theorem 3.22]{RaeWil98}. It can also be proved by using  the TRO $M_+\oplus M_-^{op}$ and appealing to
 \cite[Prop.\ 3.8]{BunTimQJM} (in general, if ($M,[\cdot,\cdot,\cdot])$ is a C*-ternary ring,  then $M^{op}$ denotes the C*-ternary ring $(M,-[\cdot,\cdot,\cdot])$).
.

We now consider the corresponding property in the context of C*-ternary rings. We begin by modifying some earlier notation.  Recall that if $M$ is a C*-ternary ring, $\ell(f,g)$, for $f,g\in M$, is the element of $B(M\oplus M)$ defined by $\ell(f,g)(x,y)=([fgx],[gfy])$. To emphasize the dependence on $M$ we denote $\ell(f,g)$ by $\ell_M(f,g)$.
Thus
\[
L(M):=\overline{sp}\{\ell_M(f,g): f,g\in M\}\subset B(M\oplus M).
\]

 Let $I$ be a closed ideal in the C*-ternary ring $M$. 
Since $I$ is also a C*-ternary ring,
\[
L(I):=\overline{sp}\{\ell_I(f,g): f,g\in I\}\subset B(I\oplus I).
\]

We now define 
\[
\widehat L(I)=\overline{sp}\{\ell_M(f,g): f,g\in I\}\subset L(M)\subset B(M\oplus M).
\]
Note that 
\[
\ell_I(f,g)=\ell_M(f,g)|_{I\oplus I}.
\]

Similarly, recall that $r(f,g)$, for $f,g\in M$, is the element of $B(M\oplus M)$ defined by $r(f,g)(x,y)=([xfg],[ygf)$. To emphasize the dependence on $M$ we denote $r(f,g)$ by $r_M(f,g)$.
Thus
\[
R(M):=\overline{sp}\{r_M(f,g): f,g\in M\}\subset B(M\oplus M).
\]

For a close ideal $I$  in $M$, we have
\[
R(I):=\overline{sp}\{r_I(f,g): f,g\in I\}\subset B(I\oplus I),
\]
and we define
\[
\widehat R(I)=\overline{sp}\{r_M(f,g): f,g\in I\}\subset R(M)\subset B(M\oplus M).
\]
Note that 
\[
r_I(f,g)=r_M(f,g)|_{I\oplus I}.
\]

The proof of the following proposition is a straightforward application of the definitions.
 \begin{proposition} 
 Let $I$ be a closed ideal in the C*-ternary ring $M$.
\begin{description}
\item[(a)] $\widehat L(I)$ is a closed two-sided  ideal in  the C*-algebra $L(M)$, and the map $\phi$ 
defined as $\phi(\ell_M(f,g))=\ell_I(f,g)$, for $f,g\in I$, extends to a contractive  *-homomorphism of $\widehat L(I)$ onto $L(I)$.
\item[(b)]
$\widehat R(I)$ is a closed two-sided  ideal in  the C*-algebra $L(M)$, and the map $\psi$ 
defined as $\psi(r_M(f,g))=r_I(f,g)$, for $f,g\in I$, extends to a contractive  *-homomorphism of $\widehat R(I)$ onto $R(I)$.
\end{description}
\end{proposition}

We next define
\[
\widehat{\mathcal A}(I)=\left[\begin{matrix}
\widehat L(I)& I\\
\overline{I}&\widehat R(I)
\end{matrix}\right]
\subset 
{\mathcal A}(M)=\left[\begin{matrix}
L(M)& M\\
\overline{M}&R(M)
\end{matrix}\right].
\]

The proof of the following proposition is  also a straightforward application of the definitions.
It is included for the convenience of the reader and to motivate one of the questions which follows it.

\begin{proposition}\label{prop:0410233}
Let $I$ be a closed ideal in a C*-ternary ring $M$.  Then $\widehat {\mathcal A}(I)$ is a closed self-adjoint ideal in ${\mathcal A}(M)$. The map $I\mapsto \widehat{\mathcal A}(I)$ is injective from closed ideals of $M$ to closed self-adjoint two-sided ideals of ${\mathcal A}(M)$.
\end{proposition}
\begin{proof}
We show for example that $ {\mathcal A}(I)$ is a right ideal in ${\mathcal A}(M)$.
For 
\[
\left[\begin{matrix}
A& f\\
\overline{g}&B
\end{matrix}\right]\in \widehat {\mathcal A}(I)
\hbox{ and }
\left[\begin{matrix}
A'& f'\\
\overline{g'}&B'
\end{matrix}\right]\in {\mathcal A}(M)
\]
we have, by (\ref{eq:1116231}),
\[
\left[\begin{matrix}
A&f\\
\overline{g}&B
\end{matrix}\right]
\left[\begin{matrix}
A'&f'\\
\overline{g'}&B'
\end{matrix}\right]
=\left[\begin{matrix}
AA'+\ell(f,g')&A\cdot f'+f\cdot B'\\
\overline{g}\cdot A'+B\cdot\overline{g'}&r(g,f')+B\circ B'
\end{matrix}\right].
\]
and it is required to show that 
\begin{equation}\label{eq:0411231}
AA'+\ell(f,g')\in \widehat L(I),
\end{equation}
\begin{equation}\label{eq:0411232}
A\cdot f'+f\cdot B'\in I
\end{equation}
\begin{equation}\label{eq:0411233}
\overline{g}\cdot A'+B\cdot\overline{g'}\in\overline{I}
\end{equation}
\begin{equation}\label{eq:0411234}
r(g,f')+B\circ B'\in \widehat R(I).
\end{equation}

To prove (\ref{eq:0411231}) and  (\ref{eq:0411232}), we may assume that $A=\ell_M(g,h)$ for $g,h\in I$, $A'=\ell_M(g',h')$ for $g',h'\in M$, and $B'=r_M(g'',h'')$ for $g'',h''\in M$. By \cite[Lemma 1.1(ii)]{PluRusJMAA}, 
\[
AA'=\ell_M(g,h)\ell_M(g',h')=\ell_M(g,[h'g'h])\in \widehat L(I).
\]
Write $f=[ggg]$ for some $g\in I$ (see footnote~3). Then $\ell_M(f,g')=\ell_M(g,[g'gg])\in \widehat L(I)$, proving (\ref{eq:0411231}).
Since $A\cdot f'=[ghf']\in I$ and $f\cdot B'=[fg''h'']\in I$, this proves  (\ref{eq:0411232}). (\ref{eq:0411233}) and  (\ref{eq:0411234}) are proved similarly. 
\end{proof}

We close this section by using the above notation to prove the last statement  of Theorem~\ref{prop:0212231}.

\begin{proposition}\label{prop:1117231}
An anti-C*-algebra is linearly homeomorphic to a C*-algebra.
\end{proposition}
\begin{proof}
We consider the anti-C*-algebra ${\mathcal A}(M_-)$ corresponding to an arbitrary C*-ternary ring with Zettl decomposition $M=M_+\oplus M_-$.
The map
\[
{\mathcal A}_0(M_-)\ni
\left[
\begin{matrix}
\sum_i\ell_{M_-}(x_i,y_i)&z\\
\overline{w}&\sum_jr_{M_-}(u_j,v_j)
\end{matrix}\right]
\mapsto
\left[
\begin{matrix}
\sum_i\ell_{M_-^{op}}(x_i,y_i)&z\\
\overline{w}&\sum_jr_{M_-}^{op}(u_j,v_j)
\end{matrix}\right]
\in
{\mathcal A}_0(M_-^{op})
\]
is bijective and bounded, so extends to a bicontinuous bijection between  ${\mathcal A}(M_-)$ and the C*-algebra ${\mathcal A}(M_-^{op})$.
\end{proof}

\begin{center}{\bf Some questions}\end{center}


\begin{question}
Which C*-algebras can appear as the linking algebra of a TRO?  Which semisimple Banach algebras with approximate identities and with continuous involution can appear as anti-C*-algebras? \end{question}

\begin{question}
In what sense is the decomposition in Theorem~\ref{prop:0212231} unique? In particular,
if $M$ and $N$ are C*-ternary rings, with *-isomorphic normed standard embeddings, does it follow that $M$ and $N$ are isomorphic?
\end{question}

  \begin{question}
If $M$ is a C*-ternary ring, is the map $I\mapsto \widehat {\mathcal A}(I)$ in Proposition~\ref{prop:0410233} from closed ideals of $M$ to closed self-adjoint two-sided ideals of ${\mathcal A}(M)$ surjective?   In the special case that $M$ is a TRO (and hence ${\mathcal A}(M)$ is a C*-algebra), this has been proved in \cite[Prop.\  2.7]{Skeide}. 
\end{question}

\begin{question}
If $M$ is a C*-ternary ring, then its bidual $M^{**}$ is a C*-ternary ring (\cite[Theorem 2]{LanRus83}). What is the relation between $L(M)^{**}$ and $L(M^{**})$, between $R(M)^{**}$ and $R(M^{**})$, between ${\mathcal A}(M)^{**}$ and ${\mathcal A}(M^{**})$?\end{question}

\section{Ternary Operator Categories---Revisited}\label{sec:0706231}

\subsection{Adjustments to  \cite[Section 4]{PluRusJMAA}}

In what follows, we shall use some notation and some results from  \cite{PluRusJMAA}, making precise references to  \cite{
PluRusJMAA} when necessary.
We first modify, as suggested in \cite{PluRusJMAA}, 
the definition of linking category. Recall that the morphism sets $(X,Y)_{\mathcal C}$ in a linear ternary category are associative triple systems.

\begin{definition}\label{def:1117231}
(Modification of \cite[Def. 3.8]{PluRusJMAA})
Given a linear ternary category $\calC$ (\cite[Def.5.1]{PluRusJMAA}), the {\it linking category} $A_{\calC}$  of $\calC$ is as follows. The objects of the category $A_{\calC}$ are the same as the objects of $\calC$.  The morphism set  $\operatorname{Hom}(X,Y)\ (=(X,Y)_{A_{\mathcal C}})$ is defined to be $\calA(X,Y)\  (={\mathcal A}((X,Y)_{\mathcal C})$, with composition  as follows. If $a\in \operatorname{Hom}(X,Y)$ and $b\in \operatorname{Hom}(Y,Z)$, then
$b\circ a$ must be $0$ unless $X=Y=Z$, in which case $b\circ a$ is defined to be the product $ab$ in $\calA(X,X)$. 
\end{definition}

\begin{remark}\label{rem:0420231}
The category $A_{\calC}$ can be considered as a ternary category 
under the composition $[abc]=ab^\#c$ and, by \cite[Lemma 1.3]{PluRusJMAA},  we obtain an injective linear ternary functor $F$ from the linear ternary category $\calC$ to $A_{\calC}$ by associating the object $X$ of $\calC$ to the object $X$ of $A_{\calC}$ and the morphism $f\in (X,Y)$ to  the morphism $\left[
\begin{matrix}
0&f\\
0&0
\end{matrix}\right]\in\calA(X,Y)_{A_{\calC}}.
$
\end{remark}

Recall that the morphism sets $(X,Y)$ in a T*-category (\cite[Definition 4.6]{PluRusJMAA}) are C*-ternary rings (\cite[Section 2]{PluRusJMAA}), so they satisfy Zettl's representation theorem \cite[Theorem 2.1]{PluRusJMAA}:
$$(X,Y)=(X,Y)_+\oplus (X,Y)_-.$$

\begin{definition}\label{def:0212231}
Given a T*-category $\calC$, the {\it positive (resp. negative) linking category} $A_{\calC}^{\pm}$  of $\calC$ is as follows. The objects of the category $A_{\calC}^{\pm}$ are the same as the objects of $\calC$.  The morphism set  $\operatorname{Hom}(X,Y)$ is defined to be $\operatorname{Hom}(X,Y)=\calA((X,Y)_{\pm})$, with composition  as follows. If $a\in \operatorname{Hom}(X,Y)$ and $b\in \operatorname{Hom}(Y,Z)$, then
$b\circ a$ must be $0$ unless $X=Y=Z$, in which case $b\circ a$ is defined to be the product $ab$ in $\calA((X,X)_{\pm})$. 
\end{definition}

Thus the  linking categories $A_{\calC}^{\pm}$ of the T*-category $\mathcal C$
  are subcategories of the linking category $A_{\calC}$  of the linear ternary category $\mathcal C$ (Definition~\ref{def:1117231}), and  $
  A_{\calC}=A_{\calC}^+\oplus A_{\calC}^-,
    $
    by \cite[Remark 4.13(ii)]{PluRusJMAA}.

\begin{theorem}\label{thm:0217231}
(Replacement for \cite[Theorem 4.11]{PluRusJMAA})
If $\calC$ is a T*-category  then $A_{\calC}^+$ is a C*-category (\cite[Definition 4.1]{PluRusJMAA}) and there is a faithful T*-functor $F$  from $\calC_+$ to $A_{\calC}^+$, the latter considered as a T*-category. 
\end{theorem}
\begin{proof}
It is clear that $A_{\calC}^+$, as defined in Definition~\ref{def:0212231}, is a linear non-unital category which, when considered as a ternary category, satisfies (ii), (iii) and (vi) in 
\cite[Def. 4.1]{PluRusJMAA}.
   Items (i), (iv), and (v)  of that definition
are tantamount  to the morphism sets of  $A_{\calC}^+$, namely, $\calA((X,Y)_+)$,  being normed as  C*-algebras. This fact is immediate from Theorem~\ref{prop:0212231}. The faithful functor is the restriction of the functor defined in Remark~\ref{rem:0420231}.  
\end{proof}

Let $\rho=(\rho_0, \{\rho_{X,Y}\})$ be a linear ternary functor from a linear ternary category $\calC$ to a a linear ternary category $\calD$. We write $\rho=(\rho_0, \{\rho_{X,Y}:X,Y\hbox{ objects  of } \calC\})$, where $\rho_0$ maps objects of $\calC$ to objects of $\calD$, and $\rho_{X,Y}$ is a linear transformation from $(X,Y)_{\calC}$ to $(\rho_0(X),\rho_0(Y))_{\calD}$ satisfying 
\[
\rho_{X,W}(h\circ g^*\circ f)=\rho_{Z,W}\circ\rho_{Z,Y}(g)^*\circ \rho_{X,Y}(f),
\]
 for $f\in (X,Y)_{\calC},\  g\in (Z,Y)_{\calC},\hbox{ and } h\in (Z,W)_{\calC}$.
 
 If $\calC$ is a T*-category, 
 T*-subcategories $\calC_{\pm}$ are defined as follows. By Zettl's representation theorem, we have $
(X,Y)=(X,Y)_+\oplus (X,Y)_-
$  for each pair of objects $X,Y$ of $\calC$. The objects of 
$\calC_{\pm}$ are the same as the objects of $\calC$, and for such objects $X,Y$,
$$(X,Y)_{\calC_{\pm}}:=(X,Y)_{\pm}.$$
 It is clear that $\calC$ is isomorphic to $\calC_+\oplus \calC_-$ and that $A_{\calC}$ is isomorphic to
 $\calA_{C_+}\oplus \calA_{C_-}$ (cf.\  \cite[Remark 4.13(ii)]{PluRusJMAA}).

\begin{remark}
(Replacement for \cite[Remark 4.12]{PluRusJMAA})
If $\rho_0$ is injective,  then there is a linear functor $\widehat\rho$  from $A_{\calC}$ to $A_{\calD}$ which extends $\rho$. In particular, if $\calC$ and $\calD$ are T*-categories, then every T*-functor  from $\calC_+$ to $\calD_+$, which is injective on objects, extends to a C*-functor
  from $A_{\calC}^+$ to $A_{\calD}^+$.
\end{remark}
\begin{proof}
Recall that $(X,Y)_{\calC}$ is an associative triple system, and $\rho_{X,Y}$ is a homomorphism of 
$(X,Y)_{\calC}$ to $(\rho_0(X),\rho_0(Y))_{\calD}$, so by \cite[Lemma 2.6]{PluRusJMAA}, it extends to a *-homomorphism $\widehat\rho_{X,Y}={\mathcal A}(\rho_{X,Y})$ from $(X,Y)_{A_{\calC}}$  to 
$(\rho_0(X),\rho_0(Y))_{A_{\calD}} $.

Thus $\widehat\rho=(\widehat\rho_0, \{\widehat\rho_{X,Y}: X,Y\hbox{ objects of }A_{\calC}\})$, where $\widehat\rho_0(X)=\rho_0(X)$,  is the desired linear  functor from $A_{\calC}$ to $A_{\calD}$ whose restriction to ${\mathcal A}_{\mathcal C}^+$ is a C*-functor to ${\mathcal A}_{\mathcal D}^+$.
\end{proof}

\begin{theorem}
(Replacement for \cite[Theorem 4.14]{PluRusJMAA})
Let $\calC$ be a T*-category.  Then there is  a faithful T*-functor  from $\calC_+$ to the T*-category $\calH$ of Hilbert spaces and continuous linear maps (\cite[Example 4.2, Remark 4.8]{PluRusJMAA}).
\end{theorem}
\begin{proof} By \cite[Theorem 4.5]{PluRusJMAA}, there is a faithful C*-functor $G_{+}$ from $A_{\calC_{+}}$ to $\calH$.  With $A_{\calC_{+}}$ considered as  a T*-category, we have that $G_{+}$ is a T*-functor from $A_{\calC_{+}}$ to $\calH$. By Theorem~\ref{thm:0217231}, there is a faithful T*-functor $F_{+}$ from $\calC_{+}$  to $A_{\calC_{+}}$, and it suffices to consider $ H=G_+\circ F_+$.
\end{proof}

Let $V$ be a C*-ternary ring. By Theorem~\ref{prop:0212231}, $V$ is the off-diagonal corner of a Banach algebra  $\calA(V)$ 
(a C*-algebra  if $V$ is a TRO) with continuous involution,   where
\[
\calA(V)=\left[\begin{matrix}
L&V\\
\overline{V}&R
\end{matrix}\right] ,
\]  
and $L=L(V)$ and $R=R(V)$ are C*-algebras.  Consider
\begin{equation}\label{eq:0915201}
\widetilde{\mathcal A}(V)=\left[\begin{matrix}
M(L)&V\\
\overline{V}&M(R)
\end{matrix}\right],
\end{equation} 
where $M(L)$ and $M(R)$ are the multiplier algebras of $L$ and of $R$.
Zettl has shown in \cite[Prop. 4.9]{Zettl83} that if $V$ is a W*-ternary ring  (\cite[Section 2]{PluRusJMAA}), then $M(R)$ and consequently $M(L)$ are W*-algebras. Therefore $\widetilde{\mathcal A}(V)$ is a unital Banach algebra with a predual and with multiplication and involution given by (\ref{eq:0906201}) and (\ref{eq:0906202}). Moreover, for a W*-ternary ring with Zettl decomposition $V=V_+\oplus V_-$, the components $V_{\pm}$ are also W*-ternary rings.

Recall that a T*-category is  said to be 
a {\it TW*-category} if each morphism set is a dual Banach space, and that for objects $X,Y$ in a TW*-category, $(X,Y)$ is a W*-ternary ring.


\begin{definition}
\label{def:0219231}
(Replacement for \cite[Def. 4.22]{PluRusJMAA})
Given a TW*-category $\calC$ (\cite[Definition 4.6]{PluRusJMAA}), the {\it linking W*-category} $\widetilde A_{\calC}^+$  of $\calC$ is as follows. The objects of the category $\widetilde A_{\calC}^+$ are the same as the objects of $\calC$.  The morphism set  $\operatorname{Hom}(X,Y)$ is defined to be $\widetilde \calA((X,Y)_+)$, as in (\ref{eq:0915201}) with $V=(X,Y)_+$, and with composition  as follows. If $a\in \operatorname{Hom}(X,Y)$ and $b\in \operatorname{Hom}(Y,Z)$, then
$b\circ a$ must be $0$ unless $X=Y=Z$, in which case $b\circ a$ is defined to be the product $ab$ in $\widetilde \calA((X,X)_+)$. 
\end{definition}

\begin{theorem}\label{thm:0831202}
(Replacement for \cite[Theorem 4.23]{PluRusJMAA})
If $\calC$ is a TW*-category  then $\widetilde A_{\calC}^+$ is a W*-category (i.e., a C*-category with morphism sets being dual spaces) and there is a faithful TW*-functor $F$  from $\calC_+$ to $\widetilde A_{\calC}^+$, the latter considered as a TW*-category. 
\end{theorem}
\begin{proof} 
It is clear that $\widetilde A_{\calC}^+$, as defined in Definition~\ref{def:0219231}, is a linear non-unital category which, when considered as a ternary category, satisfies (ii), (iii) and (vi) in 
\cite[Def. 4.6]{PluRusJMAA}.
   Items (i), (iv), and (v)  of that definition are tantamount  to the morphism sets of  $\widetilde A_{\calC}^+$, namely, $ \calA((X,Y)_+)$,  being normed as  W*-algebras. This fact is immediate from \cite[Prop. 4.21]{PluRusJMAA}.
   \end{proof}

\begin{theorem}
(Replacement for  \cite[Theorem 4.25]{PluRusJMAA})
Let $\calC$ be a TW*-category.  Then there is  a faithful TW*-functor  from $\calC_+$ to the TW*-category $\calH$ of Hilbert spaces and continuous linear maps.
\end{theorem}
\begin{proof} By \cite[Prop. 2.13]{GLR85}, there is a faithful W*-functor $F$ from $\widetilde A_{\calC}^+$ to $\calH$.  With $\widetilde A_{\calC}^+$ considered as  a TW*-category, we have that $G$ is a TW*-functor from $\widetilde A_{\calC}^+$ to $\calH$. By Theorem~\ref{thm:0831202}, there is a faithful TW*-functor $F$ from $\calC^+$  to $\widetilde A_{\calC}^+$, and it suffices to consider $ H=G\circ F$.
\end{proof}

\subsection{Two completed proofs for \cite{PluRusJMAA}}

\begin{remark}
(Completion of  the proof of  \cite[Rem. 4.10]{PluRusJMAA})
In the proof of  \cite[Rem. 4.10]{PluRusJMAA},  only items (i)-(iv) of the five conditions in \cite[Def. 4.6]{PluRusJMAA} were proved. The remaining item  (v) (as well as item (iv)) follows immediately from \cite[Prop.\ 4.5]{Abadie2017}.
\end{remark}

In the proof of \cite[Prop.\ 5.13]{PluRusJMAA}, it was stated incorrectly that C*-ternary rings are JB*-triples, and therefore, since JB*-triples satisfy Pelczynski's property V \cite{ChuMel97}, so do C*-ternary rings. However, as pointed out in Theorem~\ref{thm:0413231}, a C*-ternary ring $M=M_+\oplus M_-$ is (isomorphic to) a JB*-triple if and only if $M_-=0$. Thus a C*-ternary ring is a JB*-triple if and only if it is isomorphic to a TRO. Using this fact, we can prove the following proposition, which completes the proof of  \cite[Prop.\ 5.13]{PluRusJMAA}.

\begin{proposition}\label{prop:0420231}
A C*-ternary ring $(M,[\cdot,\cdot,\cdot])$ satisfies Pelczynski's property V.
\end{proposition}
\begin{proof}
According to the theorem of Zettl  \cite[Theorem 2.1]{PluRusJMAA}, there is a bounded operator $T$ on $M$ such that 
$(M,T\circ [\cdot,\cdot,\cdot])$ is isomorphic to a TRO. Therefore $(M,T\circ [\cdot,\cdot,\cdot])$ satisfies Pelczynski's property V. Since $(M,[\cdot,\cdot,\cdot])$ and $(M,T\circ [\cdot,\cdot,\cdot])$ are identical as Banach spaces, the proposition is proved.
\end{proof}

\end{document}